\documentclass[a4paper,11pt]{amsart}
\usepackage{mathabx}
\usepackage[utf8]{inputenc}

\usepackage{a4wide}
\usepackage{amsmath}
\usepackage{amsthm}
\usepackage{amsopn}
\usepackage{amsrefs}
\usepackage{url}
\usepackage{hyperref}
\usepackage[color=green!30]{todonotes}
\usepackage{import}

\usepackage{tikz}
\usetikzlibrary{fit,matrix,arrows}
\usepackage[all]{xy}

\newtheorem{theorem}[equation]{Theorem}
\newtheorem{proposition}[equation]{Proposition}
\newtheorem{lemma}[equation]{Lemma}
\newtheorem{corollary}[equation]{Corollary}
\theoremstyle{definition}
\newtheorem{definition}[equation]{Definition}
\newtheorem{example}[equation]{Example}

\theoremstyle{remark}
\newtheorem{remark}[equation]{Remark}

\theoremstyle{remark}
\newtheorem*{ack}{Acknowledgments}
\numberwithin{equation}{section}
\def\Z{\mathbb{Z}}
\def\Q{\mathbb{Q}}

\def\wreath{\boldsymbol{\iota_r}}
\def\redu{$(\mathrm{mod}\ p)$-reduction}
\def\spair{\beta}
\def\lpair{\lambda}

\newcommand{\rank}{\operatorname{rank}}

\DeclareMathOperator{\ord}{ord}

\DeclareMathOperator{\Fix}{Fix}
\newcommand{\One}{\operatorname{1 \kern-3.75pt 1}}
\DeclareMathOperator{\rk}{rk}
\DeclareMathOperator{\Syl}{Syl}
\newcommand{\Sq}{\Syl_q}
\newcommand{\ol}[1]{\overline{#1}}
\newcommand{\wt}[1]{\widetilde{#1}}
\title{Finite orthogonal groups and periodicity of links}
\author[M. Borodzik]{Maciej Borodzik}
\address{Institute of Mathematics, University of Warsaw, ul. Banacha 2,
02-097 Warsaw, Poland}
\email{mcboro@mimuw.edu.pl}

\author[P. Grabowski]{Przemysław Grabowski}
\address{Institute of Mathematics, University of Warsaw, ul. Banacha 2,
02-097 Warsaw, Poland}
\email{pg370922@students.mimuw.edu.pl}
\author[A. Król]{Adam Król}
\address{Institute of Mathematics, University of Warsaw, ul. Banacha 2,
02-097 Warsaw, Poland}
\email{adam.krol@student.uw.edu.pl}
\author[M. Marchwicka]{Maria Marchwicka}
\address{Department of Mathematics and Computer Science, Adam Mickiewicz University in Poznań,
  ul.~Umultowska 87, 61-614 Poznań, Poland }
\email{maria.marchwicka@amu.edu.pl}

\date{\today}

\begin{document}

\begin{abstract}
  For a prime number $q\neq 2$ and $r>0$ we study, whether there exists an isometry of order $q^r$ acting
  on a free $\Z_{p^k}$-module equipped with a scalar product. We investigate, whether there exists such an isometry
  with no non-zero fixed points. Both questions are completely answered in this paper if $p\neq 2,q$. As an
  application we refine Naik's criterion for periodicity of links in $S^3$. The periodicity criterion we obtain
  is effectively computable and gives concrete restrictions for periodicity of low-crossing knots.
\end{abstract}
\maketitle

\section{Introduction}
Let $L\subset S^3$ be a link. We say that $L$ is $m$-periodic if there exists an orientation-preserving 
diffeomorphism $\phi\colon S^3\to S^3$
such that $\phi^m=id$, $\phi^j\neq id$ for $j<m$, $\phi(L)=L$ and $L$ is disjoint from the rotation axis of $\phi$. One asks which links are periodic and how is the
periodicity reflected in the symmetry of link invariants.

Classical obstructions to knot periodicity include \cite{DavisLivingston,HillmanLivingstonNaik,Murasugi,Naik0,Naik,Przytycki-periodic,Traczyk,Traczyk2}.
Several new criteria appeared recently, based mostly on homological invariants, like \cite{BorodzikPolitarczyk,JabukaNaik,Politarczyk-Jones}, but
also on the study of fundamental groups \cite{Chen}.

Among known periodicity criteria, the criterion of Naik \cite{Naik} is very effective in obstructing
periodicity. 
It relies on an elementary observation that if $L$ is $m$-periodic, then the $\Z_m$-symmetry of the link induces a $\Z_m$ action on the homology of a branched cover,
and this action preserves the linking form. A non-trivial result relates the fixed point set of the action to the 
homology of the quotient. 
Naik's criterion relies
on understanding, which finite abelian groups (or which linking forms) admit an action with no non-zero fixed points; see Section~\ref{sec:naik}
for more details. The goal of this article is to strengthen Naik's criterion.

Before we state our main result, let us introduce some terminology.
\begin{definition}\label{def:epsilonq}
  Let $p$ be a prime number and $m$ an odd number coprime with $p$. We denote by $[m|p]$ the minimal positive exponent $s$ such that either $m|(p^s-1)$ 
  or $m|(p^s+1)$. We define $\eta(m)=1$ if $m|(p^{[m|p]}-1)$ and $\eta(m)=-1$ if $m|(p^{[m|p]}+1)$. 
\end{definition}
\begin{remark}
  We note that $\eta(m)=-1$ if and only if $p$ is a root of $-1$ in $\Z_m$, that is, if some power of $p$ is equal to $-1$ in $\Z_m$.
  Moreover, for any $n>0$, $m$ divides $p^{n[m|p]}-\eta(m)^n$.
\end{remark}
We recall also that if $p\neq 2$ and $k>0$, a free $\Z_{p^k}$-module of rank $n$ can be equipped with two non-equivalent symmetric bilinear forms with
values in $\Z_{p^k}$, these are distinguished by a sign $\epsilon$, which we introduce rigorously in Definition~\ref{def:epsilonform}.

Our new criterion is based on a result that gives even more understanding of cyclic group actions on linking (or bilinear) forms.
We believe
that it is of independent interest. 

\begin{theorem}\label{thm:main}
  Let $p\neq q$ be odd prime numbers. Let $B$ be a free $\Z_{p^k}$-module of rank $n>0$, for some $k>0$. Let $\spair\colon B\times B\to\Z_{p^k}$
  be a non-degenerate symmetric bilinear form. %
  Let $s$ be the maximal integer such that $q^s|(p^{2[q|p]}-1)$.
  Fix an integer $r>0$ and set
  $\wt{r}=\max(r-s,0)$.
  \begin{itemize}
    \item[(a)] There exists an isometry of $(B,\spair)$ of order $q^r$ if and only if $n\ge 2[q|p]q^{\wt{r}}+1$, or $n=2[q|p]q^{\wt{r}}$ and $\epsilon(B,\spair)=\eta(q)$.
    \item[(b)] There exists an isometry of $(B,\spair)$ of order $q^r$ with no non-zero fixed points if and only if $n=2[q|p]d$ for $d\ge q^{\wt{r}}$
      and $\epsilon(B,\spair)=\eta(q)^d$.
  \end{itemize}
\end{theorem}

The proof of Theorem~\ref{thm:main} stretches for  Sections~\ref{sec:groups} -- 
\ref{sec:action}. Sections~\ref{sec:groups} -- \ref{sec:maximal} deal with part (a), while Sections~\ref{sec:fixed} and~\ref{sec:action} take care with part (b).

In more detail, in Section~\ref{sec:groups} we recall the Wall's result on classification of symmetric bilinear forms on free $\Z_{p^k}$-modules. We recall
the definition of the index $\epsilon(B,\spair)$, which is crucial in our applications. In Section~\ref{sec:rank} we recall a classical result on the rank of the group
of isometries of a symmetric bilinear form on a $\Z_{p}$-module. We introduce an important concept of a \redu, which allows us to translate various
results on forms on $\Z_p$-modules to forms on $\Z_{p^k}$-modules for $k>1$. One of such results is the calculation of the Sylow $q$-group of isometries by Weir \cite{Weir},
stated in Section~\ref{sec:sylow}, which is shown to hold for isometries of forms on $\Z_{p^k}$-modules by Theorem~\ref{thm:reduction}.

The description of the Sylow $q$-groups in Theorem~\ref{thm:weir} is given in terms of the wreath product, see Definition~\ref{def:wreath}. In Section~\ref{sec:maximal}
we use an elementary properties of the wreath product to find the maximal $r$ such that the Sylow $q$-group of the group of isometries contains an element of order $q^r$.
Section~\ref{sec:maximal} is concluded with Theorem~\ref{21}, which contains the (a) part of Theorem~\ref{thm:main}. 

Part (b) of Theorem~\ref{thm:main} is proved in Sections~\ref{sec:fixed} and~\ref{sec:action}. It is essentially done on a case by case analysis. We begin
with showing that if the conditions $n=2[q|p]d$, $d\ge q^{\wt{r}}$ and $\epsilon(B,\spair)=\eta(q)^d$ are not satisfied, than any isometry of order $q^{\wt{r}}$ must have a
non-zero fixed point.
This is the statement of Theorem~\ref{thm:fixedexist} in Section~\ref{sec:fixed}. Section~\ref{sec:action} is devoted to a construction of an isometry with no non-zero fixed points. 
We perform the construction separately in the case $\wt{r}=r-s$ and in the case $\wt{r}=0$ (i.e. $r\le s$).
The key
results are Lemmas~\ref{lem:specialcase} and~\ref{lem:notsospecialanymore}, which show that if $n=2[q|p]q^{r-s}$ and $\epsilon(B,\spair)=\eta(q)^{q^{r-s}}$, then every isometry of order $q^{r-s}$ has only zero as its fixed point. Lemma~\ref{lem:specialcase} deals with the case of $\Z_{p}$-forms and Lemma~\ref{lem:notsospecialanymore} takes care of general $\Z_{p^k}$-forms
via the \redu.  Lemma~\ref{lem:notsospecialafterall} deals with the case $\wt{r}=0$, but the key argument is essentially the same as in Lemma~\ref{lem:specialcase}.
The three lemmas are used to give a proof of Theorem~\ref{thm:inverse}. This theorem concludes the proof of Theorem~\ref{thm:main}.

The main application of Theorem~\ref{thm:main} in our paper is the
following refinement of Naik's theorem \cite[Theorem 5]{Naik}.
\begin{theorem}\label{thm:app}
  Let $L$ be a $q^r$-periodic link with $r\ge 1$ and $q$ an odd prime. 
  Let $T_p$ be the $p$-torsion subgroup of $H_1(\Sigma_k(L);\Z)$, where $\Sigma_k(L)$ is a $k$-fold
  branched cyclic cover 
  for $k>1$ and $p\neq q$ an odd prime. 

  \begin{itemize}
    \item[(a)]
  If $T_p$ is non-trivial, then $T_p$ splits as an sum $T_{p,1}\oplus T_{p,2}\oplus\dots$ of pairwise orthogonal
  summands. A summand $T_{p,i}$ is a free $\Z_{p^i}$-module with linking form $\lpair_{p,i}$. 
  The $\Z_{q^r}$-symmetry of $L$ yields an action
  of $\Z_{q^r}$ on $T_p$ preserving the orthogonal splitting.
\item[(b)]
  Suppose the first homology of the $k$-fold cover of the quotient link $\Sigma_{k}(L/\Z_{q^r})$ has no $p$-torsion.
  If $T_{p,i}$ is non-trivial, then there exists an integer $d_i\ge q^{r-s}$ such that $\rank T_{p,i}=2[q|p]d_i$ and 
  $\epsilon(T_{p,i})=\eta(q)^{d_i}$.
  \end{itemize}
\end{theorem}
\begin{proof}
  Part (a) is proved as Proposition~\ref{prop:decomp}. To prove part (b) we use Proposition~\ref{prop:fixed} to conclude
  that $\Z_{q^r}$ acts on $T_p$ with no non-zero fixed points. Then by part (a) we know that $\Z_{q^r}$ acts on each of the $T_{p,i}$
  preserving the linking form and with no non-zero fixed points.

  Now $T_{p,i}$ is a free $\Z_{p^i}$-module and by the results of Section~\ref{sec:groups} a linking form $T_{p,i}\times T_{p,i}\to\Q/\Z$
  induces a symmetric bilinear form on $T_{p,i}$ with values in $\Z_{p^i}$. Isometries of the linking form are isometries of the 
  symmetric bilinear forms, see Lemma~\ref{lem:linkingtoortho}. Thus we obtain an action of $\Z_{q^r}$ on
  $T_{p,i}$ equipped with a symmetric bilinear form, the action is by isometries and with no non-zero fixed points. We conclude by Theorem~\ref{thm:main}.
\end{proof}

Our result extends \cite[Theorem 5]{Naik} in the following two ways.
\begin{itemize}
  \item While Naik's result deals with links admitting a $\Z_q$-action with $q$ prime, we extend it to links that admit an action of $\Z_{q^r}$ for $r>1$.
  \item Our new condition $\epsilon(T_{p,i})=\eta(q)^{d_i}$ rules out approximately half of linking forms that can appear as linking forms
    of a periodic link such that the quotient has no $p$-torison. This is a substantial strengthening of Naik's criterion, as shown in Section~\ref{sec:statistics}. 
\end{itemize}
\begin{example}
  As explained in detail in Section~\ref{sec:statistics}, the knots
  $13n3659$,
$14n908$,
$14n913$,
$14n2451$,
$14n2458$,
$14n6565$,
$14n9035$,
$14n11989$,
$14n14577$,
$14n23051$ and
$14n24618$ 
  pass Naik's periodicity criterion for period $3$, but are shown not to be $3$-periodic by Theorem~\ref{thm:app}.
\end{example}

%

Section~\ref{sec:examples} shows an explicit way of applying Theorem~\ref{thm:app} to obstruct periodicity. An elaborated example of a knot that actually passes the criterion
is given in Section~\ref{sec:concrete}. A comparison of our criterion with other periodicity criteria is given in Section~\ref{sec:statistics}.
\begin{ack}
  The authors would like to thank Jan Okniński for helpful discussions and for pointing to us the article of Weir.
  They are also grateful to Wojciech Politarczyk for fruitful conversations and to Chuck Livingston for his
  comments on the early draft of the paper.
  MB, PG, AK and MM are supported by the National Science Center grant 2016/22/E/ST1/00040.
\end{ack}

\section{Review of Naik's criterion}\label{sec:naik}
Let $L$ be an $m$-periodic link and let $L/\Z_m$ be the quotient link. For $k>1$ consider the branched covers $\Sigma_k(L)$ and $\Sigma_k(L/\Z_m)$.
We write $\Sigma(L),\Sigma(L/\Z_m)$ instead of $\Sigma_2(L)$ and $\Sigma_2(L/\Z_m)$.

For simplicity we will always assume that $k$ is such that $\Sigma_k(L)$ is a rational homology sphere. In particular,
there is a non-degenerate linking form $H_1(\Sigma_k(L);\Z)\times H_1(\Sigma_k(L);\Z)\to\Q/\Z$.

We have the following observation, see \cite[Section 2]{Naik0}.
\begin{lemma}\label{lem:topquotient}
  An action of $\Z_m$ on $S^3$ that preserves $L$ and whose fixed points are disjoint from $L$,
  lifts to an action of $\Z_m$ on $\Sigma_k(L)$. The quotient
  $\Sigma_k(L)/\Z_m$ is diffeomorphic to the branched cover $\Sigma_k(L/\Z_m)$ of the quotient link.
\end{lemma}
The following result relates Lemma~\ref{lem:topquotient} with the group action on homology.
\begin{proposition}[see \expandafter{\cite[Proposition 2.5]{Naik0}}]\label{prop:fixed}
  If $p$ is a prime number, $p\notdivides m$ and $\phi\colon \Sigma_k(L)\to\Sigma_k(L)$ is a generator of the action
  of $\Z_m$ on $\Sigma_k(L)$, then
  $\phi_*$ is an isometry of $H_1(\Sigma_k(L))$ (with respect to the linking form) and the fixed point set
  \[\Fix\phi_*\colon H_1(\Sigma_k(L))_p\to H_1(\Sigma_k(L))_p\]
	is equal to $H_1(\Sigma_k(L/\Z_m))_p$. Here the subscript $p$ denotes the $p$-primary part.
\end{proposition}
Proposition~\ref{prop:fixed} is non-trivial and the condition $p\notdivides m$, in general cannot be relaxed. For example, if $K=T(2,5)$, then $K$ is clearly $5$-periodic.
We have $H_1(\Sigma_2(K))=\Z_5$. The only action of $\Z_5$ on $\Z_5$ by isometries is trivial (because any isometry of $\Z_5$ has order $1$ or $2$), 
hence $\Z_5$ is a fixed subspace of this action. However
the quotient knot $K/\Z_5$ is trivial, $H_1(\Sigma_2(K/\Z_5);\Z)=0$.

In order to apply Proposition~\ref{prop:fixed} we use the following result.
\begin{proposition}\label{prop:decomp}
  For an odd prime number $p$, that is coprime with $m$, the group $T=H_1(\Sigma_k(L))_p$ decomposes as a sum $T_{p,1}\oplus T_{p,2}\oplus\dots$,
  where each of the $T_{p,i}$ is a free $\Z_{p^i}$-module. The decomposition is orthogonal with respect to the linking form. Moreover, if $\Z_m$
  acts on $T$ by isometries, then the decomposition can be made invariant with respect to this action.
\end{proposition}
\begin{proof}
  The result is rather standard, we present a quick proof for the reader's convenience.
  Write $T=H_1(\Sigma_k(L))_p$ and consider
  the linking form $\lpair\colon T\times T\to\Q/\Z$. 

  For an element $x\in T$ let $r(x)$ be the minimal positive integer such that $p^{r(x)}x=0\in T$.
  Take $x\in T$ for which $r(x)$ is maximal possible. As $\lpair$ is non-degenerate, there exists $y\in T$ such that $p^{r(x)-1}\lpair(x,y)\neq 0\in\Q/\Z$,
  in fact, if for all $y\in T$ we have $p^{r(x)-1}\lpair(x,y)=0$, then $p^{r(x)-1}x=0$ pairs trivially with all $y\in T$, contradicting non-degeneracy of $\lpair$.
  Note also that we have $r(y)=r(x)$.

  Consider now $p^{r(x)-1}\lpair(x,x)$, $p^{r(x)-1}\lpair(y,y)$ and $p^{r(x)-1}\lpair(x+y,x+y)$. If all three expressions are zero in $\Q/\Z$, we conclude that
  $p^{r(x)-1}\lpair(x,y)=0$ contradicting the assumptions (here we use the assumption that $2$ is invertible modulo $p$). Thus there is an element $z\in T$
  such that $r(z)$ is maximal possible and $p^{r(z)-1}\lpair(z,z)\neq 0$. 
  
  Let $T_z$ be the $\Z$-submodule of $T$ generated by $z,\phi_*(z),\dots,\phi_*^{m-1}(z)$. The number of generators of $T_z$ is equal to $n_z$, where
  $n_z$ is the minimal $n$ such that $\phi_*^n(z)$ belongs to the subgroup generated by $z,\phi_*(z),\dots,\phi_*^{n-1}(z)$. We have an isomorphism of $\Z$-modules
  $T_z\cong \Z_{p^{r(z)}}^{n_z}$, in particular $T_z$ is a free $\Z_{p^{r(z)}}$-module.

  Let $T'$ be the orthogonal complement of $T_z$ in $T$. As the linking form is invariant under $\phi_*$, and $T_z$ is invariant, $T'$ is also invariant.

  Now $T'$ has smaller number of generators than $T$ and we repeat the procedure to split off another orthogonal component from $T'$. As a result of a recursive
  application of the procedure, we present $T$ as an orthogonal sum of modules
  $T_{z_1}\oplus \dots\oplus T_{z_m}$. We set
  \[T_{p,i}=\bigoplus_{j\colon r(z_j)=i} T_{z_j}.\]
\end{proof}
Next result gives a number theoretical contribution to Proposition~\ref{prop:fixed}.
It is due to Davis \cite{Davis2}, see also \cite{Naik}.
\begin{theorem}\label{thm:fix1}
  If $q$ is a prime number different than $2$ and $p$, and $\Z_q$ acts on $T_{p,i}$ with no non-zero fixed points by isometries, then the rank of $T_{p,i}$
  as a $\Z_{p^i}$-module is a multiple of $2[q|p]$.
\end{theorem}
Theorem~\ref{thm:main} of the present paper is a generalisation of Theorem~\ref{thm:fix1}.

The condition described in Theorem~\ref{thm:fix1} requires some knowledge of fixed point set of the action of $\Z_q$. By Proposition~\ref{prop:fixed}
this fixed point set is the homology of the cover of the quotient link.
As explained in \cite{Naik0,Naik} it is often possible to check whether $H_1(\Sigma_k(L/Z_q);\Z)_p=0$ using Murasugi's criterion, which we now recall.

\begin{theorem}[Murasugi's criterion for knots, see \cite{Murasugi}]\label{thm:murasugi}
Suppose $K\subset S^3$ is a $q^r$-periodic knot with $q$ prime and $r>0$. Let $\Delta$ be the Alexander polynomial of $K$ and $\Delta'$ be the
Alexander polynomial of the quotient knot $K'=K/\Z_q$. Let $\ell$ be the absolute value of linking number of $K$ with the symmetry axis.
Then $\Delta'|\Delta$ and up to multiplication by a power of $t$ we have
\begin{equation}\label{eq:congknot}
  \Delta\equiv {\Delta'}^{q^r}(1+t+\ldots+t^{\ell-1})^{q^r-1}\bmod q.
\end{equation}
\end{theorem}
There are various versions of Murasugi's criterion for links, see \cite{Sakuma}, \cite[Theorem 1.10.1]{Turaev}, the precise
statements depend
on the action of the symmetry group on the set of components of the symmetric link.

Naik's homological criterion relies on combining Proposition~\ref{prop:fixed}, Theorem~\ref{thm:fix1} and Murasugi's criterion (Theorem~\ref{thm:murasugi}).
We will show two such combinations, both due to Naik \cite{Naik}. As Propositions~\ref{prop:naik1} and~\ref{prop:naik2} are instructive
and indicate how Theorem~\ref{thm:main} can be applied, we include short proofs, not claiming any originality.
\begin{proposition}\label{prop:naik1}
  Let now $L$ be a $q$-periodic link with $q$ prime. Let $\Delta$ be the Alexander polynomial of $L$.
  Suppose $\Delta'$ is the Alexander polynomial of the quotient. For any prime number $p\neq 2,q$,
  if $s=s(p)$ is a maximal positive integer such that $p^s$ divides $\Delta(-1)/\Delta'(-1)$,
  then $s$ is a multiple of $2[q|p]$.
\end{proposition}
\begin{proof}
  Let $T$ be the $p$-torsion part of $H_1(\Sigma(L);\Z)$. The rank of $T$ is equal to~$p^{s_0}$, where $s_0$
  is the maximal positive integer such that $p^{s_0}$ divides \[|H_1(\Sigma(L);\Z)|=|\det(L)|=|\Delta(-1)|.\]

  By Proposition~\ref{prop:fixed} the fixed point set of the $\Z_q$ action has rank equal to $|H_1(\Sigma(L/\Z_q);\Z)|=|\Delta'(-1)|$. Therefore
  the action of $\Z_q$ on $H_1(\Sigma(L);\Z)_p$ induces an action on the quotient $H_1(\Sigma(L);\Z)_p/H_1(\Sigma(L/\Z_q);\Z)_p$ with only trivial fixed points,
  which has rank $p^s$. We conclude by Theorem~\ref{thm:fix1}.
\end{proof}
\begin{proposition}\label{prop:naik2}
  Let $L$ be a $q$-periodic link and suppose $\Delta'$ is the Alexander polynomial of the quotient.
  If $p\neq 2,q$ is a prime number such that $p$ does not divide $\Delta'(-1)$, then the $p$-torsion
  part $H_1(\Sigma(L);\Z)_p$ decomposes as an orthogonal sum of modules $T_{p,1},\dots$, where each
  of the $T_{p,i}$ is a free $\Z_{p^i}$-module, whose rank is divisible by $2[q|p]$.
\end{proposition}
\begin{proof}
  The assumptions of the proposition imply via Proposition~\ref{prop:fixed} that $\Z_q$ acts on $H_1(\Sigma(L);\Z)_p$
  with no non-zero fixed points. We conclude by Theorem~\ref{thm:fix1}.
\end{proof}
While Propositions~\ref{prop:naik1} and~\ref{prop:naik2} do not exhaust potential applications of Naik's criterion, they give a very good balance between
applicability and generality. We show the result of implementing Proposition~\ref{prop:naik1} and  Proposition~\ref{prop:naik2} in Section~\ref{sec:statistics}.

\section{Linking forms and symmetric forms}\label{sec:groups}
Let $L$ be an $m$-periodic link and suppose $k>1$ is such that $\Sigma_k(L)$ is a rational homology sphere. We let $T$ denote the group $H_1(\Sigma_k(L);\Z)$.
Finally, let $\lpair\colon T\times T\to\Q/\Z$ be the linking form. 

For a prime number $p$ we write $T_p$ for the $p$-torsion part of $T$, so that $T=\bigoplus T_p$, where the sum is over all prime numbers $p$. The summands are pairwise
orthogonal with respect to the linking form. In order to study the linking form on $T_p$ in more detail,
we split $T_p$ into forms $T_{p,k}$ according to Proposition~\ref{prop:decomp}. Each of the $T_{p,k}$ is a free $\Z_{p^k}$-module
and the linking form restricts over $T_{p,k}$ to a form
\[\lpair\colon \Z_{p^k}^n\times \Z_{p^k}^n\to\Q/\Z,\]
where $n$ is the rank of $T_{p,k}$.

Let $\lpair\colon \Z_{p^k}^n\times \Z_{p^k}^n\to\Q/\Z$ be a non-degenerate linking form. As for any $x,y\in\Z_{p^k}^n$ we have
$p^kx=0$, we infer that $p^k\lpair(x,y)=0\in\Q/\Z$. It follows that
\[\lpair(x,y)=\frac{\spair(x,y)}{p^k}\]
for some $\spair(x,y)\in\Z$ well-defined modulo $p^k$. The following fact does not require a proof.
\begin{lemma}\label{lem:linkingtoortho}
	The form $\spair(x,y)\colon\Z_{p^k}^n\times\Z_{p^k}^n\to\Z_{p^k}$ is a symmetric, non-degenerate bilinear form. Moreover any automorphism $\psi$
	preserving the linking form is an isometry of the bilinear form $\spair$. Conversely, any isometry of the bilinear form $b$ gives an automorphism
	of the linking form.
\end{lemma}
Lemma~\ref{lem:linkingtoortho} gives a translation from linking forms to symmetric bilinear forms. The reverse passage is also possible, a form $\spair(x,y)$
induces a linking form $\lpair(x,y)=\spair(x,y)/p^k$. In particular the classification and the symmetries of linking forms correspond to the
classification and the symmetries of bilinear forms. 
\begin{definition}
  A $p^k$-\emph{form} or, if the context is clear, simply a \emph{form}, is a pair $(B,\spair)$, where $B$ is a free and finitely generated $\Z_{p^k}$-module
  and $\spair\colon B\times B\to\Z_{p^k}$ is a symmetric non-degenerate bilinear form. The \emph{rank} of the form is the rank of $B$ as a $\Z_{p^k}$-module.
\end{definition}

The classification of $p^k$-forms is well-known, we recall a classical result of Wall
\cite{wallquadratic}, see also \cite{kawauchi}.
\begin{theorem}\label{thm:class}
	Let $p>2$ be a prime, $k,n>0$. Any $p^k$-form of rank $n$ is isometric to one of the two following diagonal 
	forms:
	\begin{itemize}
	  \item the standard form, having $+1$ on the diagonal, i.e. $\spair(x,y)=\sum x_iy_i$, where $x=(x_1,\dots,x_n)$, $y=(y_1,\dots,y_n)$;
		\item the non-standard form, $\spair(x,y)=\tau x_1y_1+\sum_{i>1} x_iy_i$, where $\tau$ is not a square modulo $p^k$.
	\end{itemize}
\end{theorem}
\begin{remark}
  It is an immediate consequence of Theorem~\ref{thm:class}, that for every $p^k$-form $(B,\spair)$ of rank $n$ there exists a basis $x_1,\dots,x_n$ of $B$, such that
  $\spair(x_i,x_j)=c_i\delta_{ij}$ and $c_i=1$ for $i=2,\ldots,n$, while $c_1=1$ if the form is standard, and $c_1$ is not a square modulo $p^k$
  if the form is not standard.
\end{remark}
We will now introduce a notion of an index of a $p^k$-form, which distinguishes the standard form from the non-standard one. For reasons that will become clear later,
the index of a standard form is not always equal to $+1$, a correction term (denoted by $\epsilon_2$ below) depending on the rank and the prime number $p$ is needed.
\begin{definition}\label{def:epsilonform}
  If $(B,\spair)$ is a $p^k$-from of rank $n$,
  the \emph{index} $\epsilon(B,\spair)$ (written also $\epsilon(B)$ if no risk of confusion arises)
  is defined as $\epsilon_1\epsilon_2$, where $\epsilon_1=1$ if the form is standard (using the terminology of
  Theorem~\ref{thm:class}) and $\epsilon_1=-1$ otherwise. We set $\epsilon_2=-1$ if $p\equiv 3\bmod 4$ and $n\equiv 2\bmod 4$, otherwise $\epsilon_2=1$.

  If $M$ is a free $\Z_{p^k}$-module and $\lpair\colon M\times M\to\Q/\Z$ is a linking form, then the index $\epsilon(M)$ is
  the index of the $p^k$-form associated with $(M,\lpair)$ via Lemma~\ref{lem:linkingtoortho}.
\end{definition}
\begin{remark}
  In the present paper we use the index of a form to distinguish isomorphism classes of its group of isometries. By Theorem~\ref{thm:rank} below, the groups of isometries
  of forms of odd rank do not depend on the index. Our interest will be mainly the index of a form of even rank, the index of a form of an odd rank is defined 
  for completeness.
\end{remark}
\begin{lemma}\label{lem:isaddit}
  Suppose that $(B,\spair)=(B',\spair')\oplus (B'',\spair'')$ and $B'$, $B''$ have even rank. Then $\epsilon(B,\spair)=\epsilon(B',\spair')\epsilon(B'',\spair'')$.
\end{lemma}
\begin{proof}
  Clearly $\epsilon_1(B,\spair)=\epsilon_1(B',\spair')\epsilon_1(B'',\spair'')$. So we have to check that the same holds for $\epsilon_2$. This is obvious if
  $p\equiv 1\bmod 4$, because then $\epsilon_2=1$. If $p\equiv 3\bmod 4$, then we check $\epsilon_2(B,\spair)=\epsilon_2(B',\spair')\epsilon_2(B'',\spair'')$ on a (trivial) case by case basis. 
\end{proof}

\section{Ranks of groups of isometries}\label{sec:rank}
\begin{definition}
  For a form $(B,\spair)$ we write $O(B)$ to be the group of isometries of $B$. More specifically,
the group $O(p^k,n,\epsilon)$ is the group of isometries of the $p^k$-form of rank $n$ with  index~$\epsilon$.
\end{definition}
We have the following result, for which we refer to the book of Wilson \cite{Wilson}.

\begin{theorem}\label{thm:rank}
	The rank of $O(p,n,\epsilon)$ is equal to
	\begin{itemize}
	  \item $2p^m(p^2-1)(p^4-1)\dots (p^{2m}-1)$ if $n=2m+1$;
	  \item $2p^{m(m-1)}(p^2-1)(p^4-1)\dots(p^{2m-2}-1)(p^m-\epsilon)$ if $n=2m$.
	\end{itemize}
	Moreover, if $n$ is odd, the groups $O(p,n,+1)$ and $O(p,n,-1)$ are isomorphic.
\end{theorem}

We will extend Theorem~\ref{thm:rank} to compute the order of $O(p^k,n,\epsilon)$ for $k>1$.
First we need the following construction.
\begin{definition}\label{def:reduction}
  Let $(B,\spair)$ be a $p^k$-form. A \emph{\redu}{} (or just: a \emph{reduction}) of $(B,\spair)$ is a symmetric bilinear form
  $(B_{red},\spair_{red})$, where $B_{red}=B/p^{k-1}B$ is
  a free $\Z_p$-module
  and for $x,y\in B_{red}$ we set 
  \[\spair_{red}(x,y)=\left[\spair(\wt{x},\wt{y})\right]\in B/p^{k-1}B,\] where $\wt{x},\wt{y}\in B$ are lifts of $x,y$
  to $B$ and $[\cdot]$ denotes the class in the quotient.
\end{definition}
Informally, one may think of $B_{red}$ as `$B\bmod p$' and of $\spair_{red}$ as `$\spair\bmod p$'.
We have an obvious observation.
\begin{lemma}
  The form $(B_{red},\spair_{red})$ is non-degenerate. In particular a reduction of a $p^k$-form is a $p$-form.
\end{lemma}
Let $(B,\spair)$ be a $p^k$-form and $(B_{red},\spair_{red})$ be a \redu. Suppose $\phi\in O(B)$.
We define the isometry $\pi(\phi)\in O(B_{red})$ by the formula
\begin{equation}\label{eq:reductionpi}\pi(\phi)(x)=\left[\phi(\wt{x})\right],\end{equation}
where $x\in B_{red}$ and $\wt{x}$ is any lift of $x$ to $B$. Formula~\eqref{eq:reductionpi} defines
a group homomorphism
\[\pi\colon O(B)\to O(B_{red}),\]
which is also referred to as the \redu{} of an isometry.
\begin{lemma}\label{lem:hensel}
 The kernel of $\pi$ has cardinality $p^{(k-1)\binom{n}{2}}$. The map $\pi$ is a surjection. 
\end{lemma}
\begin{proof}
We focus on the case where $B$ is standard, the other case is analogous.
Choose a basis $\{e_i \}^{n}_{i=1}$ such that $\spair$ is the identity matrix in this base. By $\ol{e}_i$ we will denote
the basis of $B_{red}$ obtained by reducing the basis $e_1,\dots,e_n$.
Take $\phi_1\in O(B_{red})$.
Using the basis $\ol{e}_1,\dots,\ol{e}_n$ we represent $\phi_1$ as a matrix $\Phi_p$ with coefficients in $\Z_p$.
Choose a lift of the matrix $\Phi_p$ to a matrix $\Phi$ over $\Z$, that is, lift all the coefficients to $\Z$.
The matrix $\Phi$ defines also a linear map $\phi\colon B\to B$.
Obviously, $\phi$ depends on the choice of a lift of $\Phi_p$ to $\Phi$.

We want to find vectors $v_1,\ldots, v_n$ in $\Z_{p^k}^n$ such that $\wt{\phi}$ defined by $\wt{\phi}(e_i) = \phi(e_i) +p v_i$ is an isometry.
Note that regardless of the choice of $v_i$ we have $\pi(\wt{\phi})=\pi(\phi)$. Moreover, if the vectors $v_1,\dots,v_n$ are replaced by $v_1',\dots,v_n'$
such that $v_i-v_i'$ is a multiple of $p^{k-1}$, then $\wt{\phi}$ does not change.

The map $\wt{\phi}$ is an isometry if and only if the vectors $v_i$ satisfy the following condition for all $1\le i\le j\le n$
\begin{equation}\label{eq:isocondit}(\phi(e_i) +p v_i )^{T}(\phi(e_j) +p v_j) = \delta_{ij},\end{equation}
where $\delta_{ij}$ is the Kronecker's delta. Here and afterwards in the proof, we write $x^Ty$ for the scalar product
of $x$ and $y$, that is, for $\spair(x,y)$.

As $\phi_1$ is an isometry, we infer that $\phi(e_i)^T \phi(e_j)= \delta_{ij} +p c_{ij}$ for some $c_{ij}$, hence
\eqref{eq:isocondit} becomes
\begin{equation}
  p ( c_{ij}+ v_i^{T} \phi(e_j) + \phi(e_i)^T v_j )= 0.\label{eq:yestag}
\end{equation}
The equation \eqref{eq:yestag} taken for all $i\le j$ gives a system of $\frac{n(n+1)}{2}$ linear independent equations with $n^2$ variables: the variables are coefficients of the vectors $v_i$,
which should be considered as elements of the space over $\Z_{p^{k-1}}$, because, as it was mentioned above, changing coordinates of $v_i$ by a multiple of $p^{k-1}$
does not change $\wt{\phi}$.

The independence of the system \eqref{eq:yestag} follows easily from the independence of vectors $\phi(e_j)$.
The space of solutions has dimension $n^2-\frac{n(n+1)}{2}=\binom{n}{2}$.
Therefore the number of solutions is equal to $p^{(k-1)\binom{n}{2}}$ as desired.
\end{proof}
As a corollary we obtain the following fact.
\begin{corollary}\label{cor:rank2}
  Let $B$ be a $p^k$-form of rank $n$ and $B_{red}$ its \redu. Then  
  \[|O(B)|=p^{(k-1)\binom{n}{2}}|O(B_{red})|.\]
\end{corollary}
\section{Sylow groups of the orthogonal group $O(B)$}\label{sec:sylow}
In this section we let $q$ be a 
fixed odd prime number different than $p$.

\begin{definition}
For a group $G$, we denote by $\Sq G$ a Sylow $q$-subgroup of $G$.
\end{definition}

We have the following relation between Sylow groups of the orthogonal group of a form and Sylow groups of the orthogonal group of the \redu.
\begin{theorem}\label{thm:reduction}
  The \redu{} $\pi\colon O(B)\to O(B_{red})$ 
  takes a Sylow group $\Sq O(B)$ isomorphically to a Sylow group of $O(B_{red})$.
\end{theorem}

\begin{proof}
  To begin with $\Sq O(B)$ has rank equal to $q^s$, where $s$ is the maximum integer such that $q^s$ divides $|O(B)|$.
  Likewise the rank of $\Sq O(B_{red})$ is equal to $q^{s_1}$ for $s_1$ being the maximal positive integer such that  $q^{s_1}$ divides $|O(B_{red})|$.
  By Corollary~\ref{cor:rank2} we obtain $|\Sq O(B)|=|\Sq O(B_{red})|$.
  
  Let now $G=\pi(\Sq O(B))$, where $\pi\colon O(B)\to O(B_{red})$. Then $G$ is clearly a $q$-subgroup of $O(B_{red})$.
  The map $\pi|_{\Sq O(B)}\colon \Sq O(B)\to G$ is a homomorphism of groups. The kernel of $\pi|_{\Sq O(B)}$
  is a subgroup of $\Sq O(B)$ and also a subgroup of $\ker\pi$. But $\ker\pi$ is a $p$-group by Lemma~\ref{lem:hensel}
  and so its intersection with $\Sq O(B)$ is trivial. In particular $\pi|_{\Sq O(B)}$ is
  an isomorphism onto its image.

  By Sylow's theorem there exists a Sylow $q$-group $H$ of $O(B_{red})$, such that $G\subset H$. But $|G|=|\Sq O(B)|$
  and $|H|=|\Sq O(B_{red})|=|\Sq O(B)|$. So $|G|=|H|$ and hence $G=H$.
\end{proof}

To understand the structure of $\Sq O(B)$ we will use Weir's theorem \cite{Weir}. 
We need to recall the definition of a
 regular wreath product.

 \begin{definition}\label{def:wreath}
   Let $H$ and $G$ be finite groups. Let $K=\Pi_{g\in G} H$. Let $ \psi$ be action of $G$ on $K$ by left multiplication of indices. We call a semidirect product of $K$ and $G$ a \emph{regular wreath product} of $H$ and $G$. We denote it by $H \wreath G=K\rtimes_{\psi} G$.
\end{definition}
\begin{remark} The subscript $r$ in the symbol $\wreath$ is not a parameter. It is a shorthand for `regular'.
\end{remark}

The following result is due to Weir \cite{Weir}. Note that the original statement is for $p$-forms. By Theorem~\ref{thm:reduction},
the result carries through to $p^k$-forms for general $k\ge 1$.
\begin{theorem}[see \cite{Weir}]\label{thm:weir}
  Let $d$ be a natural number with $q$-adic expansion given by $d=a_0+a_1 q+ a_2 q^2+\ldots$. Let $(B,\spair)$ be a $p$-form of rank $2[q|p]d$
  such that $\epsilon(B)=\eta(q)^d$.
      Any Sylow $q$-subgroup of $O(B)$ is isomorphic to:
      \begin{equation}\label{eq:weir}\Z_{q^s}^{a_0}\times (\Z_{q^s}\wreath \Z_q)^{a_1}\times((\Z_{q^s}\wreath \Z_q)\wreath \Z_q)^{a_2}
\times (((\Z_{q^s}\wreath \Z_q)\wreath \Z_q)\wreath \Z_q)^{a_3}\times\ldots,\end{equation}
where $s$ is such that $q^s |p^{2[q|p]}-1$ and $q^{s+1} \notdivides p^{2[q|p]}-1$.
\end{theorem}
The general case of $(B,\spair)$ not satisfying the assumptions $\rk B=2[q|p]d$, $\epsilon(B)=\eta(q)^d$ will be reduced to Theorem~\ref{thm:weir}. The approach is implicit in \cite{Weir}, but it will be used in several places in the present paper, so we sketch in briefly below.

Let $B$ be a form of rank $2[q|p]d+R$ with $0\le R<2[q|p]$. We want to find a decomposition of $B$ into an orthogonal sum of forms $B'$ and $B''$ in such a way
that the inclusion induced map $O(B')\to O(B)$ (extending an isometry of $B'$ by an identity on $B''$) induces an isomorphism of Sylow $q$-groups
and for $B'$ the hypotheses of Theorem~\ref{thm:weir} are satisfied. 
To this end, we let $\mathbf{1}_+$ denote the one-dimensional $p$-form with $\epsilon=1$ and $\mathbf{1}_-$  denote the one-dimensional form with $\epsilon=-1$. 
We have the following possibilities:
\begin{itemize}
  \item \emph{Case 1: $R=0$ and $\epsilon(B)=\eta(q)^d$.} Then $B$ already satisfies the assumptions of Theorem~\ref{thm:weir}. We set $B'=B$ and $B''$ is trivial (zero-dimensional).
  \item \emph{Case 2: we have $R>0$.} We define $B'$ to be the unique form of rank $2[q|p]d$ for which $\epsilon(B')=\eta(q)^{d}$. We let $B''$ to be a form of rank $R$ that is a direct
    sum of $\mathbf{1}_+$ and $\mathbf{1}_-$ arranged in such a way that $\epsilon(B'\oplus B'')=\epsilon(B)$. Then $B$ and $B'\oplus B''$ have the same rank and the
    same $\epsilon$, so they are isomorphic.
  \item \emph{Case 3: we have $R=0$ but $\epsilon(B)\neq\eta(q)^d$.} We define $B'$ to be the unique form of rank $2[q|p](d-1)$ such that $\epsilon(B')=\eta(q)^{d-1}$. As in the previous
    case we choose $B'$ to be a direct sum of $2[q|p]$ forms $\mathbf{1}_+$ and $\mathbf{1}_-$ in such a way that $\epsilon(B'\oplus B'')=\epsilon(B)$.
\end{itemize}
\begin{remark}
  In the above construction we do not use the fact that $\epsilon(B'\oplus B'')=\epsilon(B')\epsilon(B'')$, notice that multiplicativity of $\epsilon$
  was proved only if the summands have even rank, this is not necessarily true in Case~2 above.
\end{remark}
\begin{lemma}\label{lem:b'sq}
  The inclusion $O(B')\to O(B)$ induces an isomorphism $\Sq O(B')\cong \Sq O(B)$.
\end{lemma}
\begin{proof}
  Let $G$ be a Sylow $q$-group of $O(B')$.
  Write $i\colon O(B')\to O(B)$ for the inclusion induced map. Then $i(G)$ is a $q$ subgroup of $O(B)$ and as such, it is contained in some Sylow group $H$ of $O(B)$.
  Now $|H|/|G|$ is equal to the maximal exponent $s$ such that $q^s$ divides $|O(B)|/|O(B')|$. However, 
  by Theorem~\ref{thm:rank} and Corollary~\ref{cor:rank2}, $q$ does not divide $|O(B)|/|O(B')|$. Therefore $|H|/|G|=1$ and $i(G)\subset H$, so $i(G)=H$.
\end{proof}
\begin{definition}\label{def:b'sq}
  The form $B'$ is called the \emph{maximal $q$-regular subform of $B$}. The form $B''$ is called the \emph{complementary} form.
\end{definition}

\section{Elements of maximal order in Sylow groups}\label{sec:maximal}
Let $G$ be a finite group. For a prime $q$ we define  $\mu_q(G)$ to be 
the maximal order of those elements of $G$, whose order is a power of $q$. We have the following result,
which is well-known to the experts in group theory. For the reader's convenience we present a proof.

\begin{lemma}\label{lem:muH}
  Let $H$ be a finite $q$-group and $\mu_q(H)=q^t$ for some $t>0$, then $\mu_q(H\wreath \Z_q)=q^{t+1}$.
\end{lemma}
\begin{proof}
Let $h$ be an element of $H$ such that $\ord(h)=\mu_q(H)$ and let $e$ be the identity of $H$. Then 
$h^{*}=((h,e,e,\ldots,e),1)\in H\wreath \Z_q$ has
order equal to $q^{t+1}$. Therefore $\mu_q(H\wreath\Z_q)\ge q^{t+1}$. 

To obtain the opposite inequality we first observe that if $G=K\rtimes_\psi H$, then
$\mu_q(G)\le\mu_q(K)\mu_q(H)$. Indeed, consider
$x=(k,h)\in K\rtimes_\psi\Z_q$, then
\[x^u=(k\psi_h(k)\dots\psi_{h^{u-1}}(k),h^u).\]
Suppose $u$ is the minimal power such that $h^u=e\in H$ and set $y=k\psi_h(k)\dots\psi_{h^{u-1}}(k)$. Then
we have $x^u=(y,e)$ for some $y$ and thus $x^{ru}=(y^r,e)$. From this it easily follows that $\mu_q(G)\le \mu_q(K)\mu_q(H)$.

Given the last inequality, by Definition~\ref{def:wreath} we write $H\wreath\Z_q$ as $K\rtimes_{\psi}\Z_q$, where $K=\Pi_{g\in\Z_q}H$. 
We have $\mu_q(K)=\mu_q(H)$, hence $\mu(H\wreath\Z_q)\le\mu_q(H)\mu_q(\Z_q)$ as desired.
\end{proof}

Lemma~\ref{lem:muH} combined with Weir's theorem (Theorem~\ref{thm:weir}) allows us to describe all spaces such that their orthogonal groups have elements 
of given order $q^r$.  
The following result
proves part (a) of Theorem~\ref{thm:main} from the introduction.

\begin{theorem}\label{21}
Let $s$ be a natural number such that $q^s |p^{2[q|p]}-1$ and $q^{s+1} \notdivides p^{2[q|p]}-1$. 
Let $B$ be a $p^k$-form. 

For $r\ge s$ the group $O(B)$ contains an element of order $q^r$ if and only if $B$ has rank $2[q|p]q^{r-s}$ and $\epsilon(B)=\eta(q)^{q^{r-s}}$, 
or $B$ has rank strictly greater than $2[q|p]q^{r-s}$.

If $r<s$, the group $O(B)$ contains an element of order $q^r$ if and only if $B$ has rank $2[q|p]$ with $\epsilon(B)=\eta(q)$, or rank $B$
is greater than $2[q|p]$.
\end{theorem}
\begin{proof}
  Let $B'$ be a maximal $q$-regular subform of $B$ as in Definition~\ref{def:b'sq}. By Lemma~\ref{lem:b'sq} we have an isomorphism
  of Sylow $q$-groups of $B$ and $B'$. 
  Therefore in the proof of Theorem~\ref{21} it is enough to restrict to the case, where  $B$ has rank $2[q|p]d$ and $\epsilon(B)=\eta(q)^d$.
  We clearly have $\mu_q(G\times H)=\max(\mu_q(G),\mu_q(H))$, hence
  by \eqref{eq:weir} we have
  \[\mu_q(O(B))=\max_{i\colon a_i\neq 0} \mu_q(\underbrace{(((\Z_{q^s}\wreath\Z_q)\wreath\Z_q)\dots)\wreath\Z_q}_{i\textrm{ times}}),\]
  where $a_i$ form the $q$-adic presentation of $d$, that is, $d=a_0+a_1q+a_2q^2+\dots$.

  By Lemma~\ref{lem:muH} we obtain that $\mu_q((\Z_{q^s}\wreath\Z_q)\dots\wreath\Z_q)=q^{s+i}$ if the wreath product is taken $i$ times.
  In particular
  \begin{equation}\label{eq:muqOB}\mu_q(O(B))=\max_{i\colon a_i\neq 0} q^{s+i}.\end{equation}
  It follows that if $r\ge s$, we have $\mu_q(O(B))\ge q^r$ if and only if some of the $a_t>0$ for $t\ge r-s$. This amounts to saying
  that $d\ge q^{r-s}$, if $r\ge s$, or $d\ge 1$ if $r<s$.
\end{proof}
\section{Fixed points of isometries of order $q^r$}\label{sec:fixed}
Let $B$ be a $p^k$-form and let $\phi\colon B\to B$ be an isometry. We study the fixed point set
\[\Fix\phi=\{x\in B\colon \phi(x)=x\}.\]

The following result gives a one way implication in Theorem~\ref{thm:main}(b).
\begin{theorem}\label{thm:fixedexist}
  Let $B$ be a $p^k$-form of rank $n=2[q|p]d+R$ with $0\le R <2[q|p]$. Let $\phi\in\Sq O(B)$ has order $q^r$.
  Write also $s$ for the maximal integer such that $q^s|p^{2[q|p]}-1$.
  \begin{itemize}
    \item[(a)] If $R>0$, then $\phi$ has a fixed subspace of rank at least $R$, more precisely
      there is an orthogonal decomposition of $B$ into submodules $B_1$ and $B_2$, such that $B_1$
      has rank divisible by $2[q|p]$ and $\phi|_{B_2}$ is the identity.
    \item[(b)] If $R=0$ and $\eta(q)^n\neq \epsilon(B)$, then $\phi$ has a fixed subspace of rank at least $2[q|p]$.
  \end{itemize}
\end{theorem}
\begin{proof}
  Write $B=B'\oplus B''$, where $B'$ and $B''$ are as in Definition~\ref{def:b'sq}. Note that $B''$ is non-trivial.
  Let $i\colon O(B')\to O(B)$ be the inclusion induced map. Recall that $i$ takes an element $\psi\in O(B')$ and extends it by the identity
  on $B''$. This implies, in particular, that $B''$ is contained in the fixed point set of any element of $O(B)$ that is in the image of $i$.

  Let $G$ be a Sylow group of $O(B')$. By Lemma~\ref{lem:b'sq} we have that $i(G)$ is a Sylow group of $O(B)$.

  Take $\phi\in O(B)$ of rank $q^r$ and let $H\subset O(B)$ be a Sylow group containing $\phi$. As all Sylow groups are conjugate, we have
  that $i(G)=gHg^{-1}$ for some $g\in O(B)$. In particular $\phi=g\phi'g^{-1}$ for some $\phi'\in i(G)$. Define
  $B_1=g B'$ and $B_2=g B''$. Then $B_1$ has rank divisible by $2[q|p]$ and $B_2$ contained in the fixed point set of $\phi$.
 %
\end{proof}

\section{Isometries with no non-zero fixed points}\label{sec:action}

The goal of this section is to provide the reverse (as opposed to Theorem~\ref{thm:fixedexist}) implication of Theorem~\ref{thm:main}.
\begin{theorem}\label{thm:inverse}
Suppose we have a $p^k$-form
$(B,\spair)$ of rank $n=2[q|p]d$ and $\eta(q)^d=\epsilon(B)$. Suppose also that $\Sq O(B)$ contains an element of order $q^r$.
There exists $\psi\in O(B)$ of order $q^r$ such that $\psi$ has no non-zero fixed points.
\end{theorem}
The remaining part of Section~\ref{sec:action} is devoted to the proof of Theorem~\ref{thm:inverse}. We first prove some auxiliary results,
then give the proof of Theorem~\ref{thm:inverse}.

The first step in the proof of Theorem~\ref{thm:inverse} is a result, which is of independent interest.
\begin{proposition}\label{prop:fixedsplit}
  Let $(B,\spair)$ be a $p$-form and let $\phi\in O(B)$ be of order $q^r$ for some $r>0$. Then there exists an orthogonal decomposition of $B$ into free $\Z_{p}$-modules $B_{fix}$ and $B_{oth}$,
  such that $\phi$ is the identity of $B_{fix}$ and $\phi$ acts with no non-zero fixed points on $B_{oth}$.
\end{proposition}
\begin{proof}
  Let $\zeta$ be a primitive root of unity of order $q^r$ and consider the ring $\Lambda=\Z_{p}[\zeta]$. By Maschke's theorem, $\Lambda$ is semisimple. On the other hand
  $B$ has a structure of a $\Lambda$-module, where multiplication by $\zeta$ corresponds the action of $\phi$. Hence, $B$ can be written as a direct sum of cyclic $\Lambda$-modules.
  Let $B_{fix}$
  be the sum of those cyclic modules on which $\zeta$ acts trivially and $B_{oth}$ be the sum of all the others. The decomposition is orthogonal, indeed
  consider the map $B_{fix}\to B_{oth}$ given by the inclusion of $B_{fix}$ to $B$ followed by the orthogonal projection to $B_{oth}$. This map has to be trivial by Schur's lemma, and
  therefore $B_{fix}$ is orthogonal to $B_{oth}$. By the construction $B_{fix}$ is the fixed point set of $\phi$ and $\phi$ acts on $B_{oth}$ with no non-zero fixed points.
\end{proof}
Proposition~\ref{prop:fixedsplit} allows us to prove the following lemma, which
is the key argument in the proof of Theorem~\ref{thm:inverse}.
\begin{lemma}\label{lem:specialcase}
  Suppose $(B,\spair)$ is a $p$-form. Set as usual $s$ to be the maximal integer such that
  $q^{s}|(p^{2[q|p]}-1)$. Assume $r\ge s$  and the rank of $B$ is $2[q|p]q^{r-s}$. 
  If $\epsilon(B)=\eta(q)^{q^{r-s}}$ then an isometry $\psi\in O(B)$ of order $q^r$ has no fixed non-zero fixed points on $B$.
\end{lemma}
\begin{proof}
  Suppose contrary, that is, take $\psi\in O(B)$ of order $q^r$ and assume it has a fixed subspace. Then by Proposition~\ref{prop:fixedsplit} we infer that $B$ splits
  as a direct sum of $B_{fix}$ and $B_{oth}$. The map $\psi$ preserves the splitting and write $\psi'$ for the restriction of $\psi$ to $B_{oth}$. Then
  $\psi'\in O(B_{oth})$ has order $q^r$. But the rank of $B_{oth}$ is smaller than $2[q|p]q^{r-s}$, so by Theorem~\ref{21} $O(B_{oth})$ cannot
  contain an element of order $q^r$. This contradiction shows that $\psi$ itself cannot have a non-zero fixed point.
\end{proof}
Proposition~\ref{prop:fixedsplit} was stated and proved only for $p$-forms, and not for general $p^k$-forms. In fact, it does not hold for
$p^k$-forms in general. As an example one can take the module $\Z_{p^3}^2$ with the standard linking form and a morphism $\phi$ given by the matrix
\begin{equation}\label{eq:Amatrix}A=\begin{pmatrix} 1-2p^2 & -2p \\ 2p & 1-2p^2\end{pmatrix},\end{equation}
Over $\Z_{p^3}$ we have $A^TA=\begin{pmatrix}1 & 0\\0 &1\end{pmatrix}$, hence $\phi$ is an isometry. The point $(p^2,p^2)$ is a non-zero fixed point of $A$, however
there is no free $\Z_{p^3}$-submodule on which $\phi$ acts trivially.

To deal with $p^k$-forms we apply again the \redu.
\begin{lemma}\label{lem:fixedpasses}
  Suppose $(B,\spair)$ is a $p^k$-form, $\psi$ is an isometry and $z\in B$ is a non-zero fixed point of $\psi$. Let $(B_{red},\spair_{red})$ be the \redu{}
  and let $\psi_{red}$ be the reduction of $\psi$. Then $\psi_{red}$ has a non-trivial fixed point.
\end{lemma}
\begin{proof}
  If the reduction  $\pi(z)$ is non-trivial, then clearly $\pi(z)$ is a fixed point of $\psi_{red}$. Otherwise, if $\pi(z)$ is trivial, then $z$ can be written
  as $z=p^\ell y$ for some $\ell>0$ and $y$ such that $\pi(y)\neq 0$. As $\psi(z)-z=0$ and $\psi$ is linear, we infer that $\psi(y)-y=p^{k-\ell}u$ for some $u\in B$.
  But this means that $\psi_{red}(\pi(y))-\pi(y)=0$. Hence $\pi(y)$ is a non-trivial fixed point of $\psi_{red}$.
\end{proof}
\begin{example}
  If $A$ is a matrix as in \eqref{eq:Amatrix} defining a morphism $\phi$ of $\Z_{p^3}^2$, then the \redu{} $\pi(\phi)$ is the identity matrix.
\end{example}
Given Lemma~\ref{lem:fixedpasses} we quickly generalize Lemma~\ref{lem:specialcase} to $p^k$-forms.
\begin{lemma}\label{lem:notsospecialanymore}
  Let $s$ be as in Lemma~\ref{lem:specialcase} and $r\ge s$.
  Let $(B,\spair)$ be a $p^k$-form of rank $2[q|p]q^{r-s}$ and $\epsilon(B)=\eta(q)^{q^{r-s}}$. Any element $\psi\in O(B)$ of order $q^r$ has no
  non-zero fixed points.
\end{lemma}
\begin{proof}
  If such an isometry $\psi$ has a non-trivial fixed point, then $\pi(\psi)$ has a non-trivial fixed point (and hence a fixed subspace) 
  by Lemma~\ref{lem:fixedpasses}. The isometry $\pi(\psi)$ has the same order $q^r$ because by
  Theorem~\ref{thm:reduction} the reduction operation yields an isomorphism on Sylow $q$-groups. But this contradicts Lemma~\ref{lem:specialcase}.
\end{proof}
The method of the proof of Lemma~\ref{lem:notsospecialanymore} gives the following simple result, which is needed to complete the proof of Theorem~\ref{thm:inverse}.
\begin{lemma}\label{lem:notsospecialafterall}
  Suppose $B$ has rank $2[q|p]$ and $\epsilon(B)=\eta(q)$. Then for any $\ell$ such that $1\le \ell\le s$, there exists an isometry of $B$
  of order $q^\ell$ with no non-zero fixed points.
\end{lemma}
\begin{proof}
  We act as in the proof of Lemma~\ref{lem:specialcase}. Suppose $B$ is a $p$-form of rank $2[q|p]$. From Theorem~\ref{21} we deduce that there exists an isometry $\psi$ of
  order $q^{\ell}$. If it has a fixed subspace, the form splits as an orthogonal sum of $B_{fix}$ and $B_{oth}$, but $B_{oth}$ has rank less than $2[q|p]$ and hence
  it does not admit any isometry of order $q^\ell$ for $\ell>0$.

  For $p^k$-forms we use the same argument as in Lemma~\ref{lem:notsospecialanymore}.
\end{proof}
Now we can give a proof of Theorem~\ref{thm:inverse}.
\begin{proof}[Proof of Theorem~\ref{thm:inverse}]
  First assume that $r\ge s$. We must have $d\ge q^{r-s}$, for otherwise there is no element in $O(B)$ of order $q^r$ by Theorem~\ref{21}. 
  Let
  $B_{0}$ be the unique form of rank $2[q|p]q^{r-s}$ and $\epsilon(B_0)=\eta(q)^{q^{r-s}}$.
  Take also forms $B_1,\dots,B_{d-q^{r-s}}$ that have rank $2[q|p]$ and index $\eta(q)$. The form $\wt{B}=B_0\oplus B_1\oplus\dots B_{d-q^{r-s}}$ has the same
  rank as $B$ and, by Lemma~\ref{lem:isaddit} we have that $\epsilon(B)=\epsilon(\wt{B})=\eta(q)^{d}$. Hence $B$ and $\wt{B}$ are isometric.

  We construct now an isometry on $\wt{B}$ of order $q^r$ which does not have non-zero fixed points. It is a block sum of isometries on $B_0$ and $B_1,\dots,B_{d-q^{r-s}}$.
  On $B_0$ we use an isometry $\psi_0\colon B_0\to B_0$ with no non-zero fixed points, which is provided by Lemma~\ref{lem:notsospecialanymore}. For $j=1,\dots,d-q^{r-s}$
  we take an isometry $\psi_j\colon B_j\to B_j$ which is of order $q$ and which is provided by Lemma~\ref{lem:notsospecialafterall}. The block sum of such isometries gives an
  isometry of $B$ with no non-zero fixed points and of order $q^{r}$.

  If $r<s$, the proof is analogous. Namely, as $(B,\spair)$ has rank $2[q|p]d$ we can present it as a direct sum of $d$ copies of a form $(B_0,\spair_0)$
  such that $\epsilon(B_0,\spair_0)=\eta(q)$. The form $(B_0,\spair_0)$ admits an isometry $\phi_0$ of order $q^r$ with no non-zero fixed points by Lemma~\ref{lem:notsospecialafterall}.
  The isometry on $B$ is a direct sum of $d$ copies of $\phi_0$.
\end{proof}

\section{Computing $\epsilon(B)$ for forms on the double branched cover}\label{sec:examples}
In order to apply effectively Theorem~\ref{thm:app} we need a way to compute the indices $\epsilon(B)$ of forms associated with linking forms on 
branched covers of links. We present now such an algorithm. A detailed example is given in Section~\ref{sec:concrete}.

Suppose $L$ is a link and $S$ a Seifert matrix. Write $A=S+S^T$. If $\det(A)\neq 0$, then the double branched cover $\Sigma(L)$ is a rational homology sphere.
For the rest of this section we shall assume that this is the case. Let $n$ denote the size of matrix $A$. Then
\[H_1(\Sigma(L);\Z)\cong \Z^n/A\Z^n.\]
Under this identification, the linking form is given by
\begin{equation}\label{eq:pairing}\begin{split}
  \lpair_A\colon\Z^n/A\Z^n\times \Z^n/A\Z^n&\to\Q/\Z\\
  (x,y)&\mapsto x^TA^{-1}y.
\end{split}\end{equation}
  Put $A$ into the Smith Normal Form, that is, write
\[A=CDE,\]
where $C$ and $E$ are invertible over $\Z$ and $D$ is diagonal with integer entries $(d_1,\dots,d_n)$ on the diagonal
such that $d_i|d_{i+1}$ for $i=1,\dots,n-1$.

The following simple result gives an effective way of computing the linking form on $\Sigma(L)$.
\begin{lemma}
  The pairing \eqref{eq:pairing} is isometric to the pairing
  \begin{equation}\label{eq:pairingD}\begin{split}
    \lpair_D\colon\Z^n/D\Z^n\times \Z^n/D\Z^n&\to\Q/\Z\\
(x,y)&\mapsto x^TC^TE^{-1}D^{-1}y.
\end{split}\end{equation}
\end{lemma}
\begin{proof}
  The map $\Z^n\xrightarrow{x\mapsto Cx}\Z^n$ descends to an isomorphism of abelian groups from $\Z^n/D\Z^n$ to $\Z^n/A\Z^n$.
  Now if $x,y\in\Z^n/D\Z^n$, then $Cx,Cy\in\Z^n/A\Z^n$ and we have $\lpair_A(Cx,Cy)=(Cx)^T(CDE)^{-1}Cy=\lpair_D(x,y)$.
\end{proof}

Let $p$ be an odd prime. Let $0=\alpha_0\le \alpha_1\le\dots$ be the integers such that if $\alpha_k\le i<\alpha_{k+1}$, then $p^k|d_i$
but $p^{k+1}\notdivides d_i$. Let $h_i$ to be the vectors having $0$ at all places except for the $i$-th coordinate, which is
equal to $d_i/p^k$. Define $T_{p,k}$ for the submodule of $\Z^n/D\Z^n$ generated by $h_{\alpha_k},\dots,h_{\alpha_{k+1}-1}$.
Then $T_{p,k}$ is a free $\Z_{p^k}$-module.
The sum $T_{p,1}\oplus T_{p,2}\oplus\dots$ generate the $p$-torsion part of the module $\Z^n/D\Z^n$. 
The notation $T_{p,k}$ reminds of Proposition~\ref{prop:decomp}, we remark that
in the present case the summands do not have to be pairwise orthogonal.

On $T_{p,k}$ we have the linking form that is a restriction of the linking form $\lpair_D$. Write $\spair_{p,k}$ for the associated bilinear form.
In the basis $h_{\alpha_k},\dots,h_{\alpha_{k+1}-1}$, the form can be written by a matrix, which we denote $B_k$. Note that the $ij$-entry
of $B_k$ is given by $p^k\lpair_D(h_{i+\alpha_{k}-1},h_{j+\alpha_k-1})$. 

Write $\wt{T}_{p,k}$ for the \redu{} of $T_{p,k}$ and let $\wt{B}_k$ be the matrix of the reduced form. Clearly we have
\[\det\wt{B}_k\equiv \det B_k\pmod p.\]

\begin{proposition}\label{prop:algorithm}
  The linking form on the $p^k$-torsion part of $H_1(\Sigma(K);\Z)$ has $\epsilon_1=+1$ (see Definition~\ref{def:epsilonform})
  if and only if $\det \wt{B}_k$ is a square modulo $p$.
\end{proposition}
\begin{proof}
  We first give the proof under an extra assumption that $T_{p,k}$ is orthogonal to $T_{p,k'}$ 
  with respect to the form $\lpair_D$, for all $k,k'$ such that $k\neq k'$. 
  By Theorem~\ref{thm:class} there exists a matrix $U$ with coefficients in $\Z_{p^k}$ such that $UB_kU^{T}$ is a diagonal matrix, whose all the 
  diagonal entries but the first one are equal to $1$,
  and the top-left entry is either $1$ (if $\epsilon_1=1$) or a non-square modulo $p^k$, if $\epsilon_1=-1$. Thus, the index $\epsilon_1$ depends on whether $\det UB_kU^T$ is
  a square modulo $p^k$ or not. But $\det B_k$ and $\det UB_kU^T$ differ by a square. Therefore $\epsilon_1=1$ if and only if $\det B_k$ is a square modulo $p^k$. By Hensel's lemma
  this is the same as saying that $\det \wt{B}_k$ is a square modulo $p$.
  
  Now consider the general case, where $T_{p,k}$ are not necessarily orthogonal. Write $T_p=T_{p,1}\oplus T_{p,2}\oplus\dots$.
  By Proposition~\ref{prop:decomp}, 
  $T_p$ splits as a sum $T'_{p,1}\oplus T'_{p,2}\oplus \dots$ of pairwise orthogonal summands (with respect to the linking form $\lpair_D$) and such that $T'_{p,j}$
  is a free $\Z_{p^j}$-module. We have thus an isometry 
  \[\psi\colon T_{p,1}\oplus T_{p,2}\oplus \dots \to T'_{p,1}\oplus T'_{p,2}\oplus\dots.\]
  Write $\psi_{ij}$ for the part of $\psi$ mapping from $T_{p,i}$ to $T'_{p,j}$. Choose bases $e_{i1},\dots,e_{ik_i}$, respectively $f_{jk_j},\dots,f_{jk_j}$ of $T_{p,i}$, respectively
  $T'_{p,j}$ (for $e_{ij}$ we can take vectors $h_d$ defined before, but here we need a more concise notation). With respect to these bases we denote by
  $\kappa_i$ the determinant of the intersection form on $T_{p,i}$ and $\kappa'_j$ for the determinant of the intersection form on $T_{p,j}'$.
  \begin{lemma}\label{lem:congruence}
    There is a congruence $\kappa'_i\equiv q_i^2\kappa_i\pmod{p}$ for some $q_i$ invertible modulo $p$.
  \end{lemma}
  Given Lemma~\ref{lem:congruence} we quickly finish the proof of Proposition~\ref{prop:algorithm}. Namely,
  \begin{itemize}
    \item The first part of the proof of Proposition~\ref{prop:algorithm} applied to $T'_{p,i}$ tells that $\epsilon_1=1$ if and only if $\kappa'_i$ is a square modulo $p$;
    \item Lemma~\ref{lem:congruence} implies that $\epsilon_1=1$ if and only if $\kappa_i$ is a square modulo $p$;
    \item The bases $e_{i1},\dots,e_{ik_i}$ and $h_{\alpha_i},\dots,h_{\alpha_{i+1}-1}$ differ by an invertible matrix $U_i$. Hence $\kappa_i=\det U_i\det B_i\det U^T_i$,
      so $\kappa_i=\det B_i(\det U_i)^2$. That is, $\kappa_i$ is a square modulo $p$ if and only if $\det B_i$ is a square modulo $p$.
  \end{itemize}

\end{proof}
It remains to prove Lemma~\ref{lem:congruence}.
\begin{proof}[Proof of Lemma~\ref{lem:congruence}]
  Take $x,y\in T_{p,i}$. We have
  \[p^i\lpair_D(x,y)=p^i\lpair_D(\psi x,\psi y)=p^i\lpair_D(\sum_j \psi_{ji}x,\sum_k \psi_{ki}y).\]
  The spaces $T_{p,j}'$ are pairwise orthogonal, so we rewrite the last sum as
  \begin{equation}\label{eq:plambda} p^i\lpair_D(x,y)=p^i\sum_{j}\lpair_D(\psi_{ji}x,\psi_{ji}y).\end{equation}
  \begin{lemma}\label{lem:newlemma}
    If $j\neq i$ we have 
  \begin{equation}\label{eq:jsmall}
    p^i\lpair_{D}(\psi_{ji}x,\psi_{ji}y)\equiv 0\pmod{p}.
  \end{equation}
\end{lemma}
\begin{proof}[Proof of Lemma~\ref{lem:newlemma}]
  Suppose $j<i$.
  We have $\psi_{ji}x,\psi_{ji}y\in T'_{p,j}$, which is a $p^j$-torsion part. In particular
  $\lpair_D(\psi_{ji}x,\psi_{ji}y)=\frac{c}{p^j}$ for some integer $c$. As $i>j$ we obtain \eqref{eq:jsmall} immediately.
  
  Suppose $j>i$. Both $x,y$ are annihilated by $p^i$. Consequently, $\psi_{ji}x,\psi_{ji}y$ are annihilated by $p^i$. As $T'_{p,j}$ is a free $\Z_{p^j}$-module
  we infer that $\psi_{ji}x=p^{j-i}x_j$, $\psi_{ji}y=p^{j-i}y_j$ for some elements $x_j,y_j\in T'_{p,j}$. We also have $\lpair_D(x_j,y_j)=\frac{c'}{p^j}$ for $c'\in\Z$. Hence
  \[\lpair_D(\psi_{ji}x,\psi_{ji}y)=p^{2j-2i}\lpair_D(x_j,y_j)=c'p^{j-2i}.\]
  Therefore $p^i\lpair_D(\psi_{ji}x,\psi_{ji}y)=c'p^{j-i}\equiv 0\pmod{p}$, so \eqref{eq:jsmall} holds also for $j>i$.
\end{proof}
Coming back to the proof of Lemma~\ref{lem:congruence}, from Lemma~\ref{lem:newlemma} we obtain
  \begin{equation}\label{eq:psiisometry} p^i\lpair_D(x,y)\equiv p^i\lpair_D(\psi_{ii}x,\psi_{ii}y)\pmod{p}.\end{equation}
  Write $\wt{T}_{p,i}$, respectively $\wt{T}_{p,i}'$ for the  reductions of the forms on $T_{p,i}$ and $T_{p,i}'$. Equation~\eqref{eq:psiisometry} implies
  that $\psi_{ii}$ induces an isometry between $\wt{T}_{p,i}$ and $\wt{T}_{p,i}'$. 
  As the forms $\wt{T}_{p,i}$ and $\wt{T}_{p,i}'$ are isometric, their determinants differ by a square modulo $p$.
\end{proof}

We conclude this section with the following remark. In the proof of Proposition~\ref{prop:algorithm} we did not assume that the splitting of $T_p=T_{p,1}\oplus\dots$
is invariant with respect to the group action. In fact, in the proof of Proposition we showed that the sign of the linking form on the $p^k$-torsion part of $H_1(\Sigma(K);\Z)$
does not depend on the particular choice of splitting of $T_p$ into a sum $T_{p,1}\oplus T_{p,2}\oplus\dots$.

\section{Example. Knot $10_{123}$}\label{sec:concrete}
Our aim is now to illustrate the algorithm described in Section~\ref{sec:examples} on a concrete knot. There are many knots
that pass Naik's obstructions but fail to the obstruction provided by Theorem~\ref{thm:app}, see Section~\ref{sec:statistics}. However
all such knots that are known to us have Seifert matrices of size at least $14\times 14$. In order to make the discussion in
this section more transparent, we decided to provide an example of a knot that actually is 5-periodic. Its Seifert matrix has size only $8\times 8$. We will show,
how did we verify that it passes our criterion.

\smallskip
Consider the knot $K=10_{123}$. It is well known to be 5-periodic, a 5-periodic diagram can be found in \cite[Figure 2]{JabukaNaik} or on the
KnotInfo webpage \cite{knotinfo}. 
The Alexander polynomial of $K$ is equal to
\[\Delta=t^8 - 6t^7 + 15t^6 - 24t^5 + 29t^4 - 24t^3 + 15t^2 - 6t + 1=(t^4 - 3t^3 + 3t^2 - 3t + 1)^2.\]
There are two  factors of the Alexander polynomial over $\Z[t,t^{-1}]$. One is $1$, the other one is $t^4-3t^3+3t^2-3t+1$.
We check that the latter does not satisfy \eqref{eq:congknot} for any $\ell$. On the other hand, we have a congruence
\[\Delta\equiv (1+t+t^2)^4\bmod 5,\]
so $\Delta$ (rather unsurprisingly) passes the Murasugi criterion. The polynomial $1$ is the only candidate for the Alexander polynomial of the quotient. 
In particular, by Proposition~\ref{prop:fixed}, the $\Z_5$ symmetry should act on $H_1(\Sigma(K);\Z)_p$ with only $0$ as a fixed point for all $p\neq 5$.

We study the homology of the double branched cover. The Seifert matrix of $K$ is
\[
 S= \begin{pmatrix}
-1 & 0 & 0 & 0 & 0 & 0 & 0 & 0 \\
1 & -1 & 0 & 0 & 0 & 0 & 0 & 0 \\
0 & 1 & -1 & 0 & 0 & 0 & 0 & 0 \\
0 & 0 & 1 & -1 & 0 & 0 & 0 & 0 \\
-1 & 1 & 0 & 0 & 1 & -1 & 0 & 0 \\
0 & -1 & 1 & 0 & 0 & 1 & -1 & 0 \\
0 & 0 & -1 & 1 & 0 & 0 & 1 & -1 \\
0 & 0 & 0 & -1 & 0 & 0 & 0 & 1
\end{pmatrix}.
\]
We have
\begin{equation*}
  A=S+S^T=CDE,
\end{equation*}
where $D$ is diagonal matrix with the diagonal vector $(1,1,1,1,1,1,11,11)$ and $C$ and $E$ are invertible over $\Z$.
We see that the homology of the double branched cover is equal to $\Z_{11}^2$. Now for $p=11$ and $q=5$ we have $[q|p]=1$, because $5|(11-1)$. In particular $10_{123}$ passes the Naik's criterion. Note also, for future
use, that $\eta(q)=1$. 

We calculate the index of the linking form on $\Z_{11}^2$. The generators for the module $\Z^8/D\Z^8$ are $h_1=(0,0,0,0,0,0,1,0)$ and $h_2=(0,0,0,0,0,0,0,1)$.
The matrix of the linking form in this basis is given by the $2\times 2$ square submatrix of $C^TE^{-1}D^{-1}$ in the bottom left corner:
\[
\begin{pmatrix}

\begin{tikzpicture}[every node/.style={minimum width=1.5em}]
\matrix (m1) [matrix of math nodes]
{%
-46 & -20 & -60 & -52 & -62 & -51 & -4 & -5 \\
-20 & -12 & -35 & -34 & -39 & -31 & -3 & -2 \\
-60 & -35 & -98 & -93 & -107 & -85 & -8 & -6 \\
-52 & -34 & -93 & -90 & -103 & -81 & -8 & -5 \\
-62 & -39 & -107 & -103 & -118 & -93 & -9 & -6 \\
-51 & -31 & -85 & -81 & -93 & -74 & -7 & -5 \\
-4 & -3 & -8 & -8 & -9 & -7 & -\frac{8}{11} & -\frac{4}{11} \\
-5 & -2 & -6 & -5 & -6 & -5 & -\frac{4}{11} & -\frac{6}{11} \\};
\fill[opacity=0.3,yellow] (m1-7-7.north west) -- (m1-7-8.north east) -- (m1-8-8.south east) -- (m1-8-7.south west) --cycle;
\end{tikzpicture}
\end{pmatrix}.
\]
The associated bilinear form $(B,\spair)$ on $\Z_{11}^2$ has matrix
\[
  \begin{pmatrix}
    -8 & -4\\
    -4 & -6
  \end{pmatrix}.
\]
The determinant of this matrix is equal to $32\equiv 10\bmod 11$, so it is not a square modulo $11$, so the form has $\epsilon_1=-1$.
However, the rank $n=2$ is congruent to $2\bmod 4$ and $11\equiv 3\bmod 4$, so according to the definition of the index (Definition~\ref{def:epsilonform}) we have
$\epsilon_2=-1$, so $\epsilon(B)=1$.
Then $\epsilon(B)=\eta(q)$. The knot passes our criterion.

\section{Theorem~\ref{thm:app} for low crossing knots}\label{sec:statistics}
We applied the following criteria for knots up to 15 crossings and for periods $q=3,5,7,11,13$. The Sage script that we used is
available in \cite{maria}.
\begin{itemize}
  \item Przytycki's criterion for HOMFLYPT polynomials \cite{Przytycki-periodic};
  \item Murasugi's criterion as stated in Theorem~\ref{thm:murasugi};
  \item Naik's homological criterion in the way stated in Propositions~\ref{prop:naik1} and~\ref{prop:naik2};
  \item Theorem~\ref{thm:app}.
\end{itemize}
The way we applied these criteria is as follows. For a given knot $K$ and period $q$ we checked the Przytycki's criterion.
Independently, we
were
looking for all candidates $\Delta'$ for the Alexander polynomial of the quotient, that is, those that satisfy the statement of Theorem~\ref{thm:murasugi}.
For each such $\Delta'$ we checked 
if the criterion of Proposition~\ref{prop:naik1} is satisfied for all prime numbers $p$ different than $2$ and $q$ that are also factors of $\Delta(-1)$.
If $\Delta'$ passed the criterion of Proposition~\ref{prop:naik1} we checked if $K$ passes Proposition~\ref{prop:naik2} with this $\Delta'$ for
all prime numbers $p\neq 2,q$ that divide $\Delta(-1)$ but do not divide $\Delta'(-1)$.

If $\Delta'$ passed also the criterion of Proposition~\ref{prop:naik2} we were looking at Theorem~\ref{thm:app} for all prime numbers $p\neq 2,q$ such that
$p|\Delta(-1)$ but $p$ does not divide $\Delta'(-1)$.

If at some moment $\Delta'$ did not pass the criterion, it was discarded from a list of potential Alexander polynomials of the quotient. If the list
was empty, we concluded that the knot $K$ is not $q$-periodic. We recorded, if this conclusion is achieved using Proposition~\ref{prop:naik1} only,
or one needs Proposition~\ref{prop:naik2} or even Theorem~\ref{thm:app} to obstruct $q$-periodicity.

It turns out that for periods greater than $5$, Theorem~\ref{thm:app} did not obstruct any case that was not obstructed by combined Murasugi's and Naik's criteria. 

For period $5$, our criterion obstructed knots $14n26993$, $15a80526$, $15n83514$ and $15n95792$, but all of these knots are not $5$-periodic by Przytycki's criterion.

The most interesting situation was for period $3$. Here we were able to obstruct knots $12a100$ and $12a348$, which can also be obstructed using Jabuka and Naik's $d$-invariants
criterion, see \cite[Section 2.4]{JabukaNaik} for a detailed discussion of $12a100$ knot. 

There are also 19 alternating knots with crossing numbers from 13 to 15, whose $3$-periodicity is obstructed by Theorem~\ref{thm:app}, but not by Naik's criterion
and also not by Przytycki's criterion. These are:

\smallskip
\begin{tabular}{lllll}
$  13a4648$&
$  14a7583$&
$  14a7948$&
$14a8670$&
$14a9356$\\
$14a14971$&
$14a16311$&
$14a17173$&
$14a17260$&
$14a18647$\\
$15a6030$&
$15a6066$&
$15a10622$&
$15a15077$&
$15a33910$\\
$15a36983$&
$15a46768$&
$15a72333$&
$15a82771$.
\end{tabular}

Among non-alternating knots with $12$--$15$ crossings, there are 57 knots whose $3$-periodicity can be obstructed by Theorem~\ref{thm:app}, but for which Naik's criterion
and Przytycki's critierion do not obstruct $3$-periodicity. These are
$13n3659$,
$14n908$,
$14n913$,
$14n2451$,
$14n2458$,
$14n6565$,
$14n9035$,
$14n11989$,
$14n14577$,
$14n23051$ and
$14n24618$, as well as further $46$ knots with $15$ crossings. 

We note that for these non-alternating examples the Jabuka--Naik criterion \cite{JabukaNaik} cannot be easily applied, because it requires calculating $d$-invariants of double branched covers of non-alternating knots,
for which we do not have any algorithm so far. Another criterion involving knot homology, namely Khovanov homology, see \cite{BorodzikPolitarczyk}, does not work for period $3$, so it
cannot obstruct periodicity of these $57$ knots.

We have also applied Theorem~\ref{thm:app} to obstruct $3^2$-periodicity of knots, but we could not find any example, where Theorem~\ref{thm:app} could obstruct periodicity
of any knot that passed the Murasugi's criterion for period $9$. To find
interesting examples, one probably needs to consider knots with much more crossings.
\bibliographystyle{natplain}
\def\MR#1{}
\bibliography{research2}
\end{document}